\documentclass{article}
\usepackage[utf8]{inputenc}
\usepackage{amsmath}
\usepackage{mathtools}
\usepackage{amssymb}
\usepackage{amsthm}
\usepackage{amsfonts}
\usepackage{geometry}
\usepackage{xcolor}
\usepackage{pgf,tikz}
\usepackage{fancyhdr}
\usepackage{sectsty}
\usepackage{graphicx}
\usepackage{dsfont}
\usepackage{enumitem}
\usepackage{setspace}
\usepackage{lipsum}
\usepackage{gensymb}
\usepackage{graphicx}
\usepackage{caption}
\usepackage{subcaption}
\usepackage{mathdots}
\usepackage{hyperref}
\usepackage{multicol}
\usepackage{parskip}
\usepackage{float}
\usepackage{caption}

\theoremstyle{plain}
\newtheorem{theorem}{Theorem}
\newtheorem{lemma}[theorem]{Lemma}
\newtheorem{proposition}[theorem]{Proposition}

\theoremstyle{definition}
\newtheorem{definition}[theorem]{Definition}

\theoremstyle{remark}
\newtheorem{remark}[theorem]{Remark}

\title{A New Bound on Odd Multicrossing Numbers of Knots and Links}
\author{Anshul Guha}
\date{October 2020}

\begin{document}

\maketitle

\begin{abstract}

An $n$-crossing projection of a link $L$ is a projection of $L$ onto a plane such that $n$ points on $L$ are superimposed on top of each other at every crossing. We prove that for all $k \in \mathbb{N}$ and all links $L$, the inequality $$c_{2k+1}(L) \geq \frac{2g(L) + r(L)-1}{k^2}$$ holds, where $c_{2k+1}(L)$, $g(L)$, and $r(L)$ are the $(2k+1)$-crossing number, $3$-genus, and number of components of $L$ respectively. This result is used to prove a new bound on the odd crossing numbers of torus knots and generalizes a result of Jablonowski (see \cite{unknown}).

We also prove a new upper bound on the $5$-crossing numbers of the 2-torus knots and links. Furthermore, we improve the lower bounds on the $5$-crossing numbers of $79$ knots with $2$-crossing number $ \leq 12$. Finally, we improve the lower bounds on the $7$-crossing numbers of $5$ knots with $2$-crossing number $\leq 12$.

\end{abstract}

\section{Introduction}

The modern study of knots first originated in the field of chemistry. In the 1800s, it was commonly believed that all of space was filled with a substance called luminiferous ether. Lord Kelvin attempted to explain the existence of the distinct elements of the periodic table by conjecturing that each element corresponded to a different type of knot in the ether. Kelvin's theory inspired Peter Guthrie Tait and C. N. Little to try to tabulate all the possible knots. By doing so, they hoped to create a ``periodic table" where each knot corresponded to a different element. 

In 1887, the Michelson-Morley experiment disproved the existence of the luminiferous ether, rendering Kelvin's theory incorrect. However, the tables of knot tabulations were still useful to mathematicians who studied knots for their own sake. Since then, mathematicians have devised increasingly sophisticated and powerful invariants to classify and distinguish knots. Some of these invariants include the crossing number, bridge number, knot genus, Alexander-Conway Polynomials, the Jones Polynomial, the HOMFLY-PT polynomial, and more. 

In recent years, knots have become useful to molecular biologists. For example, knot theory can be used to determine the chirality of a molecule. Additionally, the typical DNA helix has a similar structure as the 2-torus knots. Furthermore, molecular knots, known as knotanes, appear naturally in the shapes of the $3_1$, $4_1$, $5_1$, and $6_2$ knots of the Rolfsen Knot Table. Knotanes have also been synthetically synthesized in the shapes of the $3_1$, $4_1$, $5_1$, and $8_{19}$ knots. Due to their globular shapes, synthetic knotanes have the potential to be the building blocks in nanotechnology. 

A {\em knot projection} is a picture of the projection of a knot into a plane. A crossing is formed whenever two points on the knot are superimposed upon each other in a knot projection.

\begin{figure}[H]
    \begin{center}
        \captionsetup{width=.7\linewidth}
        \includegraphics[scale=1.0]{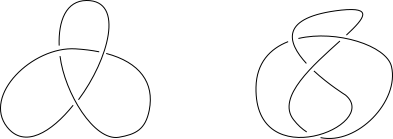}
        \caption{Left: A projection of the trefoil knot $3_1$. Right: A projection of the figure-eight knot $4_1$.}
    \end{center}
\end{figure}

We study {\em multicrossing knot projections}, a generalization of knot projections, where possibly more than two points on the knot are superimposed upon each other at every crossing. In an $n$-crossing projection of a knot, exactly $n$ points on the knot are superimposed upon each other at every crossing.

\begin{figure}[h!]
    \begin{center}
        \captionsetup{width=.8\linewidth}
        \includegraphics[scale=0.8]{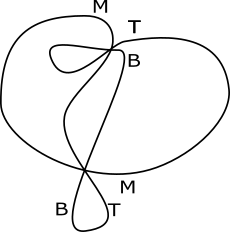}
        \caption{A 3-crossing projection of the trefoil knot. $T$ denotes the topmost strand at the crossing, $M$ denotes the middle strand, and $B$ denotes the bottom strand.}
    \end{center}
\end{figure}

A single $n$-crossing of a knot can be {\em decomposed} into $\frac{n(n-1)}{2}$ $2$-crossings by slightly moving each of the strands until no three strands intersect at the same point. The decomposition of an $n$-crossing is not unique; for example, two possible decompositions of a $4$-crossing are shown in Figure 3:

\begin{figure}[H]
    \begin{center}
        \includegraphics[scale=0.8]{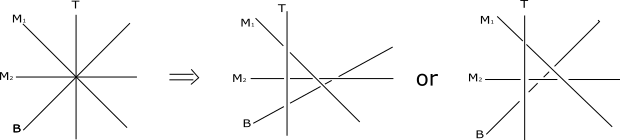}
        \caption{A 4-crossing where $T$ is the top strand, $M$, is the strand below $T$, $M_2$ is the strand below $M_1$, and $B$ is the bottom strand. The two figures on the right are two ways of decomposing a $4$-crossing into six $2$-crossings.}
    \end{center}
\end{figure}

The \textit{$n$-crossing number} of a knot $K$, denoted by $c_n(K)$, is the least number of $n$-crossings required in an $n$-crossing projection of a knot. Due to  Lemma 4.3 of \cite{2012arXiv1207.7332A}, every knot has an $n$-crossing diagram for all $n$, so the value $c_n(K)$ is always well-defined for all $n \geq 2$. Many papers have proved various bounds on $c_n(K)$, either for specific $n$ or for all knots $K$ in general (see \cite{2012arXiv1207.7332A, 2014MPCPS.156..241A,2014arXiv1407.4485A,2017arXiv170609333A}).

Because each $n$-crossing is equivalent to a collection of $\frac{n(n-1)}{2}$ $2$-crossings, a trivial bound on $c_n(K)$ is given by
\begin{equation}\label{eq:trivialbound}
c_n(K) \geq \frac{2c_2(K)}{n(n-1)}.
\end{equation}
This follows from the decomposition of each of the $c_n(K)$ $n$-crossings into $\frac{n(n-1)}{2}$ $2$-crossings.
In \cite{2014arXiv1407.4485A}, the authors used the \textit{Kauffman bracket polynomial} $\langle K \rangle$ to improve this bound to $$c_n(K) \geq \frac{\text{span} \langle K \rangle}{\left\lfloor{\frac{n^2}{2}}\right\rfloor + 4n-8}.$$

Here, $\text{span} \langle K \rangle$ denotes the difference between the largest exponent and smallest exponent of the variable of $\langle K \rangle$.

In \cite{unknown}, the author proves that for any knot or link $L$, the inequality $c_{3}(L) \geq 2g(L) + r(L) - 1$ holds, where $g(L)$ and $r(L)$ denote the $3$-genus and number of components of $L$ respectively. We extend the work of \cite{unknown} to prove 

\begin{theorem}\label{thm:thm1}
\textit{For any knot or link L and any positive integer $k$, we have the inequality}
$$c_{2k+1}(L) \geq \frac{2g(L) + r(L)-1}{k^2}.$$
\end{theorem}

The rest of the paper is organized as follows. In Section $2$, we present the necessary definitions and develop some tools for the proof of Theorem \ref{thm:thm1}. We prove Theorem \ref{thm:thm1} in Section $3$. Section $4$ covers our findings on an upper bound on the $5$-crossing number of the $2$-torus knots.

\section{Preliminaries}

An {\em n-crossing diagram} of a projection of a link $L$ is a planar graph created from an n-crossing projection of $L$ (see, e.g., Figure \ref{Figure:trefoil}). Every vertex of an $n$-crossing diagram has degree $2n$.

\begin{figure}[H]
    \begin{center}
        \captionsetup{width=0.8\linewidth}
        \includegraphics[scale=0.9]{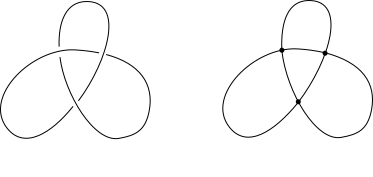}
        \caption{A projection of the trefoil knot is to the left. Its corresponding diagram is to the right. In this example, $n = 2$.}
        \label{Figure:trefoil}
    \end{center}
\end{figure} 

\begin{definition}
An \textit{odd-crossing projection} of a link $L$ is an $n$-crossing projection of $L$ where $n$ is odd. 
\end{definition} 

\begin{definition}
A \textit{piecewise natural orientation} of an odd-crossing diagram of a link $L$ is an orientation on each edge of the diagram such that the boundary of each face of the diagram is oriented (clockwise or counterclockwise). A projection $P$ of $L$ combined with an orientation on each component of $P$ is a {\em piecewise natural projection} if the orientation of the corresponding diagram $D$ of $P$ with the same orientation on each edge of $D$ as on $P$ is a piecewise natural orientation.
\end{definition}

\begin{figure}[H]
    \captionsetup{width=0.8\linewidth}
    \begin{center}
        \includegraphics[scale=0.5]{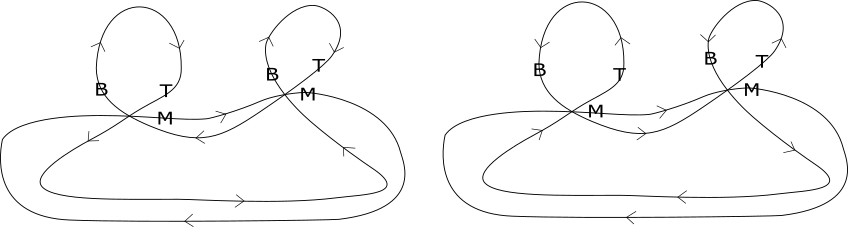}
        \caption{Both of these diagrams are diagrams of the $T_{4,2}$ torus link. The diagram on the left is a piecewise natural orientation and the diagram on the right is not a piecewise natural orientation.}
    \end{center}
\end{figure} 

\begin{proposition}[{\cite{2017arXiv170609333A}}]
\textit{It is always possible to impose a piecewise-natural orientation on an odd-crossing diagram $D$ of a link $L$.}
\end{proposition}


Now, we extend Lemmas $1$ and $2$ of \cite{2017arXiv170609333A} to prove

\begin{lemma}\label{lemma:lemma5}
An orientation of a diagram $D$ of a link $L$ is piecewise-natural iff the edges around every vertex alternate in an in-out-in-out  fashion, as in the following diagram.
\end{lemma}
\begin{figure}[H]
    \captionsetup{width=0.8\linewidth}
    \begin{center}
        \includegraphics[scale=0.5]{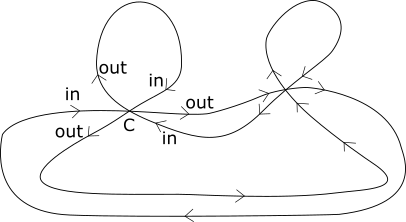}
        \caption{In this diagram of the torus link $T_{4,2}$, there are two vertices which correspond to two $3$-crossings in the corresponding projection of $P$ of $T_{4,2}$. The orientations of the six edges around each vertex alternate in an in-out in-out in-out fashion.}
    \end{center}
\end{figure}
\begin{proof}
First, we prove the forward direction. Consider a piecewise-natural diagram $D$ of $L$, and assume for the sake of contradiction that there exists a crossing $C$ where two adjacent edges $a$, $b$ are oriented inward toward $C$. Then the face $F$ that is partially bounded by the union of $a$ and $b$ cannot be oriented either clockwise or counterclockwise, because one of $\{a, b\}$ is oriented clockwise with respect to $F$ and the other is oriented counterclockwise with respect to $F$. A similar proof applies for the case were $a, b$ are oriented outwards from $C$. 

Next, we prove the reverse direction. Assume that in a knot diagram $D$, every crossing is oriented such that the edges around every crossing alternate. Consider an arbitrary face $F$ of $D$, which is bounded by some vertices and some edges within $D$. If we traverse the boundary of $F$, starting from an arbitrary vertex and following the orientation of each edge, there is no vertex at which we are forced to stop, since that would imply that two adjacent edges bounding $F$ both point inwards or outwards. So, it is possible to traverse the entire boundary of $F$, eventually returning to our starting vertex. This implies that the face $F$ is oriented either counterclockwise or clockwise, which finishes the proof.
\end{proof}

Note that the definition of a piecewise natural projection breaks down for even-crossings, because an in-out-in-out pattern in an even-crossing would imply opposite orientations for the two halves of each individual strand in the crossing. 

\begin{definition}
A \textit{Seifert surface} of a link $L \in \mathbb{R}^3$ is an embedded compact surface in $\mathbb{R}^3$ whose boundary is $L$. 
\end{definition}

\begin{definition}
The $3$-genus $g(L)$ of a link $L$ is the least possible  genus realized by any Seifert surface of $L$. 

\end{definition}

{\em Seifert's algorithm}, which takes a $2$-crossing projection of a link $L$ and outputs a Seifert surface of $L$, is described as follows (see Figure \ref{Figure:Seifert}):

\begin{enumerate}
    \item Give each component of $L$ an orientation.
    \item Notice that at every crossing, each strand goes into the crossing and then comes out of the crossing. Connect the part of each strand that goes into the crossing to the part of the other strand which comes out of the crossing.
    \item There will now be a set of oriented circles on the plane. These circles are called Seifert circles. 
    \item Connect the Seifert circles to each other at the original crossings of $L$ using twisted bands.
\end{enumerate}

\begin{figure}[H]
    \captionsetup{width=0.8\linewidth}
    \begin{center}
        \includegraphics[scale=0.8]{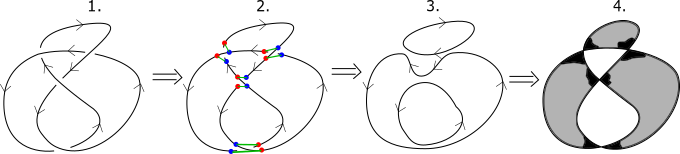}
        \caption{Seifert's algorithm for the $4_1$ knot in the Rolfsen knot table. For a strand at a crossing, the blue dot represents the part of the strand that goes into the crossing and the red dot represents the part of the strand that comes out of the crossing. The green bars indicate the connections that are formed in step $2$ of Seifert's algorithm.}
        \label{Figure:Seifert}
    \end{center}
\end{figure}

The \textit{canonical $3$-genus} of a link $L$, denoted by $g_c(L)$, is the least possible genus of a Seifert surface of a knot that is obtained using the Seifert's algorithm. It follows that for every link, $g_c(L) \geq g(L)$. 

In \cite{FormalKnotTheory}, Kauffman proved that Seifert's algorithm applied to {\em alternative} links (a class of links that includes all alternating links and all torus links) gives a Seifert surface of least genus. In this case, $g_c(L)=g(L)$. However, Moriah (see \cite[p. 105]{TheKnotBook}) found an infinite class of knots where $g_c(L) > g(L)$.

\begin{lemma} \label{l8}
\textit{A piecewise-natural orientation of a diagram $D$ of a link $L$ can be checkerboard-colored so that the faces of $D$ that are oriented clockwise correspond exactly to the white faces of $D$.}
\end{lemma}

\begin{proof} \label{lemma:lemma9}
The key observation is that any two faces of $D$ which share an edge cannot both be oriented clockwise or counterclockwise. Indeed, if the faces $F_1$ and $F_2$ of $D$ share an edge $E$, then the orientations of $F_1$ and $F_2$ are completely determined by the orientation of $E$. As shown in Figure \ref{figure:Figure-8}, if $E$ induces a counterclockwise orientation on $F_1$, then it induces a clockwise orientation on $F_2$. Similarly, if $E$ induces a clockwise orientation on $F_1$, then it induces a counterclockwise orientation on $F_2$.

\begin{figure}[H]
    \captionsetup{width=0.8\linewidth}
    \begin{center}
        \includegraphics[scale=0.8]{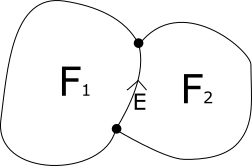}
        \caption{In this diagram, the orientation of the common edge $E$ of faces $F_1$ and $F_2$ forces $F_1$ to be is oriented counterclockwise and $F_2$ to be oriented clockwise.}
        \label{figure:Figure-8}
    \end{center}
\end{figure}

Now, color each face of $D$ that is oriented clockwise white, and color each face of $D$ that is oriented counterclockwise black. Then every face of $D$ is colored and no two adjacent faces are of the same color. So, the overall coloring imposed on the faces of $D$ is a checkerboard coloring.
\end{proof}

\begin{definition}
Let an {\em oriented checkerboard coloring} of a diagram $D$ of a link $L$ be a checkerboard coloring on $D$ for which every white face is oriented clockwise and every black face is oriented counterclockwise.
\end{definition}

Lemma \ref{l8} proves that every odd-crossing diagram $D$ has an oriented checkerboard coloring. 

\begin{lemma} \label{lemma:lemma7}
Consider a $(2k+1)$-crossing $C$ in a piecewise-natural projection $P$ of a link $L$. Impose an oriented checkerboard coloring on the corresponding diagram $D$.  Then there exists an algorithm to decompose $C$ into $k(2k+1)$ $2$-crossings so that 
\begin{enumerate}
    \item When (step $2$ of) Seifert's algorithm is applied  to these new $k(2k+1)$ 2-crossings, $k^2$ Seifert circles are formed. 
    \item In the corresponding diagram $D' \subseteq D$ of the decomposition of $C$ with the induced piecewise-natural orientation, each new Seifert circle is bounded by a white region of $D'$.
\end{enumerate}

 Moreover, with a slight modification to the algorithm, we can decompose $C$ such that Seifert's algorithm forms $k^2$ new Seifert circles, each of which is bounded by a black region of $P$ instead.
\end{lemma}

\begin{proof} 

In this proof, we refer to a projection $P$ of $L$ and its corresponding diagram $D$ interchangeably. Specifically, when we mention the color of a specific region of $P$ due an oriented checkerboard coloring on $P$, we refer to the color of the corresponding face of the oriented checkerboard coloring of the diagram $D$ induced by $P$.

Rotate $C$ such that one of the strands passing $C$ is vertical. Label the $(2k+1)$ strands passing through $C$ in counterclockwise order, with the vertical strand labeled $1$. Suppose that the region to the left of the top half of this strand is white. Due to the oriented checkerboard coloring of $D$ which induced the coloring on $P$, the boundary of this white region is oriented clockwise. So, the vertical strand must be oriented downwards.

Decompose $C$ into $k(2k+1)$ $2$-crossings by using the following algorithm: Starting with Strand $2$ and moving counterclockwise, isotope each strand upwards as shown below for $k=2$. (The process is similar for other $k$.)

\begin{figure}[H]
    \begin{center}
        \includegraphics[scale=0.5]{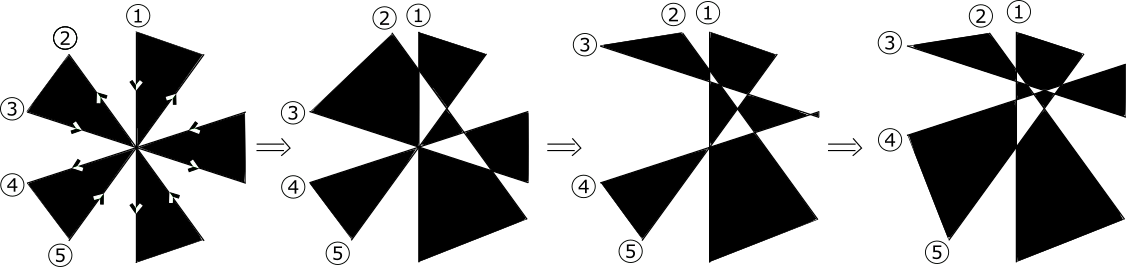}
    \end{center}
\end{figure}

We will now prove the lemma using induction. The base case($k = 1$) is done as in the following figure. Decomposing a $3$-crossing using the algorithm above creates one Seifert circle.

\begin{figure}[H]
    \begin{center}
        \includegraphics[scale=0.5]{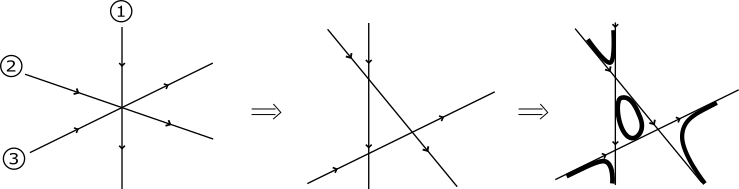}
    \end{center}
\end{figure}

For the inductive step, assume that the lemma is true for a $(2k-1)$-crossing. Add two new dotted strands to the bottom of a $(2k-1)$-crossing and orient them as shown below. The new configuration, with the added strands, is the diagram we would have produced using the algorithm above on a $(2k+1)$-crossing. 

\begin{figure}[H]
    \captionsetup{width=0.8\linewidth}
    \begin{center}
        \includegraphics[scale=0.62]{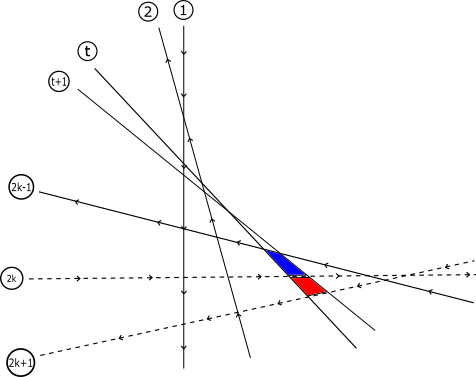}
        
    \end{center}
    \caption{Strands $1$, $2$, $2k-1$, and $2k$, $2k+1$ are drawn. Also drawn are strands $t$ and $t+1$, where $t$ is an arbitrary integer where $2 \leq t \leq 2k-2$. Note that if $t = 2$, then strands $t$ and $2$ coincide. Similarly, if $t = 2k-2$, then strands $2k-2$ and $2k-1$ coincide. }
    \label{figure:figure7}
\end{figure}

The following three paragraphs refer to  Figure \ref{figure:figure7}.

Notice that the region bounded by strands $1$, $2k-1$, and $2k$ is divided into $2k-2$ separate regions by strands $2$ to $2k-2$. $2k-3$ of these regions are quadrilaterals, and $1$ region is a triangle. Each of the $2k-3$ quadrilaterals is created by strand $t$ and $t+1$, where $1 \leq t \leq 2k-3$. Since $t$ and $t+1$ will have opposite orientations, strand $2k-1$ is oriented inwards, and strand $2k$ is oriented outwards, $t$ must be even for the region bounded by strands $t$, $t+1$, $2k-1$, and $2k$ (the blue quadrilateral) to form a cycle and thus a Seifert circle. Additionally, the triangle created by strands $2k$, $2k-1$, and $2k-2$ forms a cycle due to the $k=1$ case of the claim, which is proven above.

Similarly, the region bounded by strands $1$, $2k-1$, and $2k$ is divided into $2k-1$ bounded regions by strands $2$ to $2k-1$. Similar to the last case, each of the $2k-2$ bounded quadrilaterals is created by strand $t$ and $t+1$, where $1 \leq t \leq 2k-2$. Since $t$ and $t+1$ have opposite orientations, $2k$ is oriented outwards, and $2k+1$ is oriented inwards, $t$ must be odd for the region bounded by strands $t$, $t+1$, $2k$, and $2k+1$ (the red quadrilateral) to form a cycle and therefore a Seifert circle. Additionally, the triangle created by strands $2k+1$, $2k$, and $2k-1$ forms a cycle due to the $k=1$ case of the claim, which is proven above.

Finally, we note that $(k-1)+1=k$ new Seifert circles are created in the region bounded by strand $1$, $2k-1$ and $2k$. Similarly, $k+1$ new Seifert circles are created in the region between strand $1$, $2k$ and $2k+1$. This is a total of $2k+1$ new Seifert circles, which gives us $k^2+(2k+1) = (k+1)^2$ Seifert circles in the decomposed $(2k+1)$-crossing projection. Moreover, all the Seifert circles in the projection appear in alternating regions, so they must be of the same color. By construction, they are all white.\\

\begin{figure}[H]
    \captionsetup{width=0.8\linewidth}
    \begin{center}
        \includegraphics[scale=0.5]{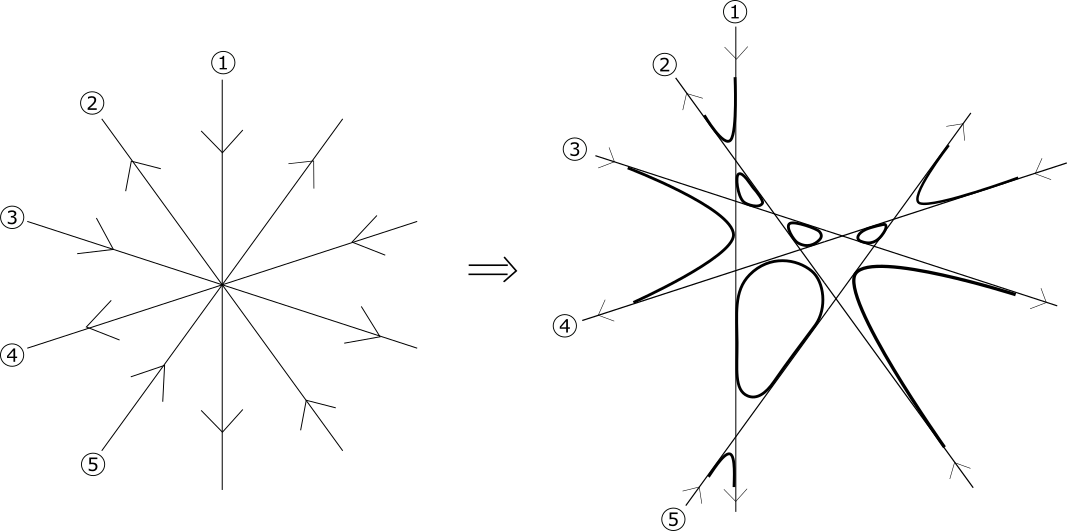}
    \end{center}
    \caption{A visual representation of the $4$ Seifert circles created after decomposing a $5$-crossing using the algorithm presented in the beginning of the proof of Lemma \ref{lemma:lemma7}. Also, notice the partial Seifert circles on the borders of the figure to the right all ``follow" the boundary of a white region.} \label{figure:Figure-10}
\end{figure}

For the ``moreover" part of the lemma, we present an algorithm that decomposes $C$ into $k(2k+1)$ crossings such that applying Seifert's algorithm forms $k^2$ new Seifert circles, each bounded by a black region of $P$. Indeed, rotate $C$ such that one of the strands passing through $C$ is vertical, with the region to the left of the top half of this strand being black. Similar to the case above, the vertical strand must be oriented upwards. Now move counterclockwise and isotope each strand upwards in a similar fashion as the case above. The diagram below shows this process for $k=2$.

\begin{figure}[H]
    \begin{center}
        \includegraphics[scale=0.55]{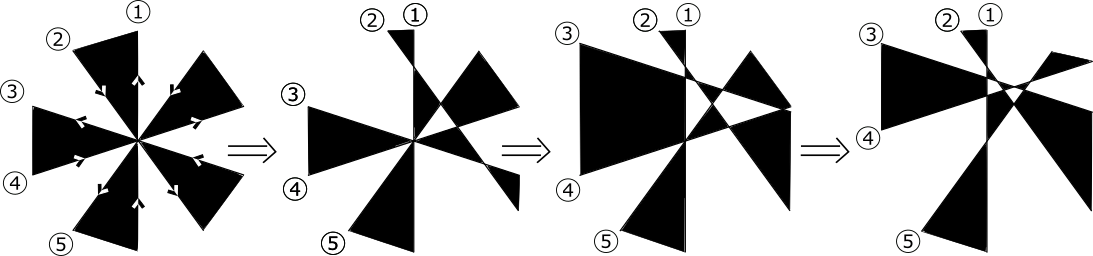}
    \end{center}
\end{figure}

The difference between this case and the previous case is that the orientation of each edge in this case is the opposite of the orientation of the corresponding edge in the previous case. Therefore, the oriented faces in this case are in one-to-one correspondence with the oriented faces in the previous case, with the only difference being that in this case, the orientations are reversed. If we color each oriented face black, $k^2$ black oriented faces are produced by this algorithm.
\end{proof}

\begin{lemma} \label{lemma:lemma12}
\textit{Assume we decompose every $(2k+1)$-crossing in a $(2k+1)$-crossing projection $P$ of a link $L$ using the algorithm (in Lemma \ref{lemma:lemma7}) that creates $k^2$ Seifert circles which correspond to the white regions of the oriented checkerboard coloring of $P$. Then the number of Seifert circles formed is equal to the number of white regions in $P$. If we instead decompose every crossing so that the $k^2$ Seifert circles correspond to the black regions of the oriented checkerboard coloring of $P$, then the number of Seifert circles formed is equal to the number of black regions in $P$.}
\end{lemma}

\begin{proof}
For every crossing in $C$, we decompose the crossing using the algorithm which creates $k^2$ new Seifert circles, each of which is bounded by a white region of an oriented checkerboard coloring of $P$. Because all the Seifert circles  ``follow" the boundaries of the white regions (see Figure \ref{figure:Figure-10}), Seifert's algorithm associates each white region in $P$ with a Seifert circle formed by its boundary. Thus, the number of Seifert circles is equal to the number of white regions in $P$. 

Similarly, we can also decompose every crossing in $C$ using the algorithm which creates $k^2$ new Seifert circles bounded by a black region of an oriented checkerboard coloring of $P$. Again, since all the Seifert circles follow the boundaries of the white regions, Seifert's algorithm associates each black region in $P$ with the Seifert circle formed by its boundary.

\end{proof}

\begin{proposition}[{\cite{MengWang}}]\label{prop:prop11}
Given a projection of a link $L$ with $r$ components, $n$ 2-crossings, and $s$ Seifert circles, applying Seifert's algorithm to the projection yields a Seifert surface of genus $$\frac{2+n+s-r}{2}.$$
\end{proposition}


\section{Proof of Theorem \ref{thm:thm1}}

Consider a planar $(2k+1)$-diagram $D$ of a link $L$, and call the corresponding projection $P$. Furthermore, assume that $D$ has $n$ vertices; each vertex in $D$ corresponds to a crossing in $P$. Each vertex in $D$ has degree $4k+2$, and each edge in $D$ is connected to $2$ vertices. Thus, by counting all the edges connected to each vertex of $D$ and then dividing by $2$ to account for the fact that every edge is counted twice, we obtain that $D$ has exactly $(2k+1)n$ edges. Now, we can apply the Euler Characteristic to find the number of faces in $D$, which is $$F = E-V+2 = (2k+1)n-n+2 = 2(kn+1).$$

Pick a piecewise-natural orientation on $P$. By Lemma \ref{l8}, we have an oriented checkerboard coloring on $D$ and $P$ such that corresponding regions in $D$ and $P$ have the same color. Let $W$ to be the number of white faces in $P$.

 At least $kn+1$ of the faces have to be either white or black. Assume $W \geq kn+1$. Now, we can use Lemma \ref{lemma:lemma7} to decompose each of the $(2k+1)$-crossings in $P$ into $2$-crossings. Each $(2k+1)$-crossing, when decomposed, adds $k^2$ white faces to $P$; so, after decomposing all $n$ crossings, $W \geq k^2n+kn+1$. Moreover, Lemma \ref{lemma:lemma12} also implies that the Seifert Circles obtained after applying Seifert's algorithm to $P$ correspond exactly with the white faces of $P$. Thus, applying Seifert's algorithm to $P$ will yield $k^2n+kn+1$ Seifert circles. Proposition
\ref{prop:prop11} gives
    \[
        \begin{split}
        g_c(L) & \leq \frac{2+k(2k+1)n-k^2n-kn-1-r(L)}{2}
        \\
        &= \frac{k^2n-r(L)+1}{2}.
        \end{split}
    \]

Since the above formula is valid for all $n$-crossing diagrams, we replace $n$ with $c_{2k+1}(L)$ to get that $$g_c(L) \leq \frac{k^2c_{2k+1}(L)-r(L)+1}{2}.$$
    
\noindent Rearranging, we obtain $$c_{2k+1}(L) \geq \frac{2g_c(L)+r(L)-1}{k^2} \geq \frac{2g(L)+r(L)-1}{k^2}.$$
This finishes the proof.

If $W<kn+1$, then we can use the ``moreover" part of Lemma \ref{lemma:lemma7} to decompose each crossing in $P$ and create more black Seifert circles. The proof would proceed in the same way, with the role of black and white colors switched.\\

\begin{remark}
In \cite{Crowell} and \cite{Scherich}, it is proved that for all knots and all alternating links $L$, we have that $\text{span}(\triangle(L)) = 2g(L)+r(L)-1$, where $\triangle(L)$ is the Alexander polynomial of $L$. Thus, the bound from Theorem \ref{thm:thm1} simplifies to $$c_{2k+1}(L) \geq \frac{2g(L)+r(L)-1}{k^2} \geq \frac{span(\triangle(L))}{k^2}.$$ The second $\geq$ sign is actually an equality for alternating links. Moreover, the bound is quite strong for knots, since equality holds for all knots with $\leq 10$ crossings. The knot with the lowest $2$-crossing number  for which this inequality breaks is the Kinoshita-Terasaka knot. The new form of the bound is preferable because computing the Alexander Polynomial of a knot is algorithmic, and computing the genus is often difficult.
\end{remark}
\begin{remark}
Using the KnotInfo website as a resource, we can use Theorem \ref{thm:thm1} to improve the lower bounds on the $5$-crossing numbers on $79$ different knots with $2$-crossing number $\leq 12$ from $2$ (given by \cite{2014arXiv1407.4485A}) to $3$. These knots are:

$11a_{367}$,
$12a_{146}$,
$12a_{369}$,
$12a_{576}$,
$12a_{716}$,
$12a_{722}$,
$12a_{805}$,
$12a_{815}$,
$12a_{819}$,
$12a_{824}$,
$12a_{835}$,
$12a_{838}$,
$12a_{850}$,
$12a_{859}$,
$12a_{864}$,
$12a_{869}$,
$12a_{878}$,
$12a_{898}$,
$12a_{909}$,
$12a_{916}$,
$12a_{920}$,
$12a_{981}$,
$12a_{984}$,
$12a_{999}$,
$12a_{1002}$,
$12a_{1011}$,
$12a_{1013}$,
$12a_{1027}$,
$12a_{1047}$,
$12a_{1051}$,
$12a_{1114}$,
$12a_{1120}$,
$12a_{1128}$,
$12a_{1134}$,
$12a_{1168}$,
$12a_{1176}$,
$12a_{1191}$,
$12a_{1199}$,
$12a_{1203}$,
$12a_{1209}$,
$12a_{1210}$,
$12a_{1211}$,
$12a_{1212}$,
$12a_{1214}$,
$12a_{1215}$,
$12a_{1218}$,
$12a_{1219}$,
$12a_{1220}$,
$12a_{1221}$,
$12a_{1222}$,
$12a_{1223}$,
$12a_{1225}$,
$12a_{1226}$,
$12a_{1227}$,
$12a_{1229}$,
$12a_{1230}$,
$12a_{1231}$,
$12a_{1233}$,
$12a_{1235}$,
$12a_{1238}$,
$12a_{1246}$,
$12a_{1248}$,
$12a_{1249}$,
$12a_{1250}$,
$12a_{1253}$,
$12a_{1254}$,
$12a_{1255}$,
$12a_{1258}$,
$12a_{1260}$,
$12a_{1273}$,
$12a_{1283}$,
$12a_{1288}$,
$12n_{242}$,
$12n_{472}$,
$12n_{574}$,
$12n_{679}$,
$12n_{688}$,
$12n_{725}$,
$12n_{888}$

Finally, Theorem \ref{thm:thm1} allows us to improve the bound on $7$-crossing number of the following knots from $1$ to $2$: $11a_{367}$, $12n_{242}$, $12n_{725}$, $12n_{472}$, $12n_{574}$. 
\end{remark}

\section{Torus Knots}
Applying Theorem \ref{thm:thm1} specifically to the $(p,q)$-torus knots gives 
\begin{equation}\label{eqn:Tpq}
    c_{2k+1}(T_{p,q}) \geq \frac{(p-1)(q-1)}{k^2}
\end{equation}
 which, for large $p$, $q$, improves on the ``trivial" bound (\ref{eq:trivialbound}) by a factor of approximately $2$. Because the span of the Kauffman Bracket Polynomials of the torus knots is $4(p+q-2)$, the bound derived in \cite{2014arXiv1407.4485A} gives worse bounds than the trivial bound. 

Of particular interest are the knots $T_{2,4t+1}$ and $T_{2,4t+3}$. In \cite{2014MPCPS.156..241A}, Adams proves that $c_4(T_{2,4t+1}) = c_4(T_{2,4t+3}) = t+1$ for all integers $t$. Moreover, (\ref{eqn:Tpq}) implies that

\begin{enumerate}
    \item $c_5(T_{2,4t+1}) \geq t$, and
    \item $c_5(T_{2,4t+3}) \geq t+1$.
\end{enumerate}

\begin{remark} \label{remark:remark12}
It is conjectured that $c_n(L) \geq c_{n+1}(L)$ for any  $L$ and any $n$. Most of the knots with $\leq 9$ crossings have been shown to have this property (see \cite{2014arXiv1407.4485A}). However, for $c_4(T_{2,2t+1}) \geq c_5(T_{2,2t+1})$ to hold, we must have 

\begin{enumerate}
    \item $c_5(T_{2,4t+1}) = t$ or $t+1$
    \item $c_5(T_{2,4t+1}) = t+1$
\end{enumerate}

This motivates further investigation into upper bounds on the $5$-crossing numbers of the $2$-torus knots and links.
\end{remark}

\subsection*{Upper Bounds on $c_5(T_{2,t})$}

We use the definition of bigon chains and crossing-covering circles in \cite{2012arXiv1207.7332A}. Let an {\em odd bigon chain} be a bigon chain of odd length.

\begin{lemma} \label{l15}
\noindent A bigon chain of length $3$ can be isotoped into a single $4$-crossing.
\end{lemma}

\begin{proof}
The proof is by construction.

\begin{figure}[H]
    \begin{center}
        \includegraphics[scale=0.35]{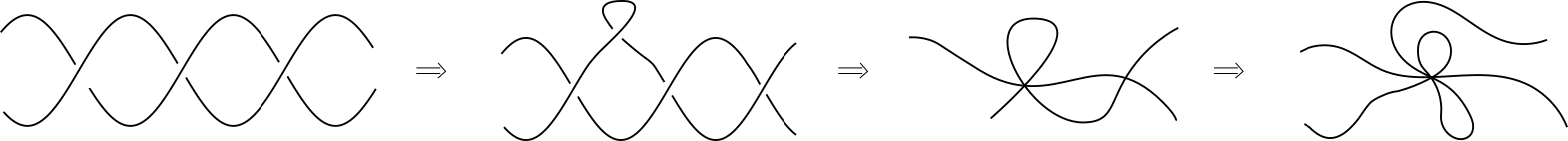}
    \end{center}
\end{figure}
\end{proof}

\begin{theorem} 
We have the following upper bound:
\[ c_5(T_{2,t}) \leq  \left\{
\begin{array}{ll}
      \lceil{\frac{t}{4}}\rceil & \text{if } t \text{ is even},\\
      \lfloor \frac{t+1}{3} \rfloor & \text{if } t \text{ is odd}. \\
\end{array} 
\right. \]
\end{theorem}

We prove this theorem by dividing it into two cases, based on the parity of $t$. 

\begin{proposition}[$t =$ even]
\textit{If $4|t$, then $c_5(T_{2,t}) \leq \frac{t}{4}$. If $4|(t+2)$, then $c_5(T_{2,t}) \leq \frac{t+2}{4}$.}
\end{proposition}

\begin{proof}
In \cite{2014MPCPS.156..241A}, it is shown that a $4$-bigon chain can be rearranged into a single $4$-crossing. We use this fact repeatedly in our construction.

If $4|t$, then each of the $\frac{t}{4}$ consecutive $4$-bigon chains in $T_{4,t}$ can be rearranged into a $4$-crossing. Then we can use a crossing-covering circle to change all the $4$-crossings into $5$-crossings.

If $4|(t+2)$, then we can divide the crossings of $T_{4,t}$ into $\frac{t-2}{4}$ sets of $4$-bigon chains, with two consecutive $2$-crossings left over. Each of the $4$-bigon chains can be rearranged into a single $4$-crossing, while the two remaining $2$-crossings can be rearranged into a single $3$-crossing. Finally, all the $4$-crossings can be converted into $5$-crossings through the use of a crossing-covering circle, and the $3$-crossing can be turned into a $5$-crossing.

These two processes are shown below for the cases $t=8$ and $t=10$.\\

\begin{figure}[H]
    \begin{center}
        \includegraphics[scale=0.7]{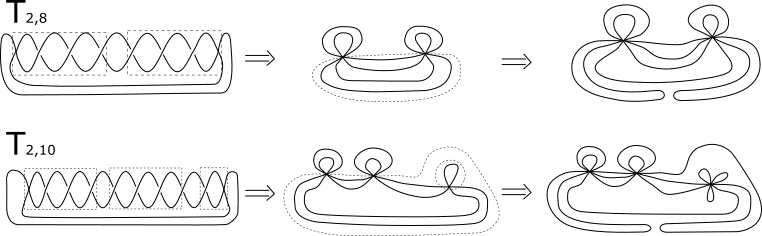}
    \end{center}
\end{figure}
\end{proof}

\begin{proposition}[$t = $ odd] \label{thm:thm14}
\textit{If $t=3q+k$ is odd, $k \in \{0,1,2\}$, then $c_5(T_{2,3q+k})$ is bounded above by the following piecewise function:}

\[ f(q) =   \left\{
\begin{array}{ll}

      q & \text{if } k = 0,1 \pmod 3\\
      q+1 & \text{if } k = 2 \pmod 3 \\
\end{array} 
\right. \]
\end{proposition}

\begin{proof}
The proof is by construction. We break it up into 3 cases:

\underline{Case 1: $q$ is divisible by $3$.} In this case, $T_{2,3q}$ is essentially equivalent to $q$ consecutive odd bigon chains, each with length $3$. Take each bigon chain of length $3$ and isotope it into a $4$-crossing (see Lemma \ref{l15}). Afterwards, draw a dotted crossing circle that passes exactly once through each of the $q$ resulting $4$-crossings, and pull the ``bottom" strand over to create a $5$-crossing projection of $T_{2,3q}$.

\underline{Case 2: $q \equiv 1 \pmod 3.$} Here, $T_{2,3q+1}$ can be broken up into $q$ consecutive odd bigon chains, each with length $3$, and $1$ remaining $2$-crossing. Take each bigon chain of length $3$ and isotope it into a $4$-crossing. Afterwards, draw a dotted crossing-covering circle that passes exactly once through each of the $q$ resulting $4$-crossings and the remaining $2$-crossing. Then pull the overstrand of the $2$-crossing around the crossing-covering circle to eliminate the remaining  $2$-crossing. We are left with a $5$-crossing projection of $T_{2,3q+1}$.

\underline{Case 3: $q \equiv 2 \pmod 3.$} Here, $T_{2,3q+2}$ can be broken up into $q$ consecutive odd bigon chains, each with length $3$, and $2$ remaining $2$-crossings. Take each bigon chain of length $3$ and isotope it into a $4$-crossing. Afterwards, draw a dotted crossing-covering circle that passes exactly once through each of the $q$ resulting $4$-crossings and the two remaining $2$-crossings. Pull the overstrands of one of the $2$-crossings around the crossing-covering circle to get $q$ total $5$-crossings and one $3$-crossing. Finally, turn the last $3$-crossing into another $5$-crossing by looping one strand over the crossing twice.

As an example, the constructions for $T_{2,5}, T_{2,7}$, and $T_{2,9}$ are shown below.

\begin{figure}[H]
    \begin{center}
        \includegraphics[scale=0.55]{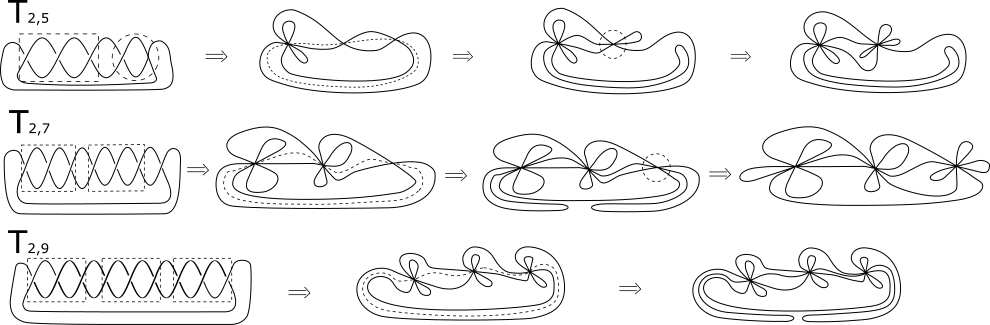}
    \end{center}
\end{figure}
\end{proof}

Proposition \ref{thm:thm14} is sharp for the $T_{2,3}$, $T_{2,5}$, and $T_{2,7}$ torus knots, which are the only $(2,q)$-torus knots (to the best of the author's knowledge) whose $5$-crossing numbers are known (see \cite{2014arXiv1407.4485A}). Additionally, combining Proposition \ref{thm:thm14} and Theorem \ref{thm:thm1}, we get that $c_5(T_{2,9}) \in \{2,3\}$, $c_5(T_{2,11}) = \{3,4\}$, $c_5(T_{2,13}) \in \{4,5\}$, $c_5(T_{2,15}) \in \{4,5\}$, etc. 

\section{Conclusion and Future Work}
In this paper, we proved a new bound on odd crossing numbers of multicrossing projections of knots and links.

In the future, one possible direction could be to develop an analog of piecewise natural projections for even crossing projections. This could be used to generalize Theorem \ref{thm:thm1} to arbitrary multicrossing projections. 

Another possible avenue of research is proving that multicrossing numbers are monotonic; i.e, $c_n(L) \geq c_{n+1}(L)$ for any $n$ and any link $L$. By Remark $\ref{remark:remark12}$, this would prove that for all $k$, $c_5(T_{4k+3,2}) = k+1$. On the other hand, in order to disprove the monotonicity of multicrossing numbers, we can try proving that $c_5(T_{11,2}) = 4$, or that Theorem \ref{thm:thm14} actually gives equality.

Finally, there are many unknown multicrossing numbers of knots with crossing number $\leq 9$ (see the table in \cite{2014arXiv1407.4485A}). It would be helpful to fill the rest of the table with exact numbers. 

\section{Acknowledgements}

My deepest gratitude goes to Dr. Subhadip Dey (Gibbs Assistant Professor, Yale University), who suggested this research topic and helped me navigate the field through good humor, patience, constant support, and knowledge. He is a source of continued encouragement who tickled my curiosity and built my confidence. 

\bibliographystyle{plain}
\bibliography{Citations}

\end{document}